\documentclass[11pt,reqno]{amsart}
\usepackage{enumerate,amsmath,amssymb, mathrsfs}

\usepackage{color}

\usepackage{latexsym, amssymb,enumerate}
\usepackage{appendix}

\newcommand{\be}{\begin{equation}}
\newcommand{\ee}{\end{equation}}
\newcommand{\beq}{\begin{eqnarray}}
\newcommand{\eeq}{\end{eqnarray}}

\def\R{{\mathfrak R}}

\def\n{\nabla}

\newtheorem{prop}{Proposition}[section]
\newtheorem{thm}[prop]{Theorem}
\newtheorem{lem}[prop]{Lemma}

\newtheorem{rema}[prop]{Remark}

\newtheorem{defi}[prop]{Definition}

\def\begeq{\begin{equation}}
\def\endeq{\end{equation}}
\def\p{\partial}

\def\e{\epsilon}

\def\R{\mathbb R}

\def\d{\delta}
\def\a{\alpha}
\def\b{\beta}

\def\om{{\omega}}

\def \ds{\displaystyle}
\def \vs{\vspace*{0.1cm}}

\def\div{{\rm div\,}}

\def\e{{\epsilon}}

\def\odot{\setbox0=\hbox{$\bigcirc$}\relax \mathbin {\hbox
to0pt{\raise.5pt\hbox to\wd0{\hfil $\wedge$\hfil}\hss}\box0 }}

\numberwithin{equation} {section}

\def\tilde{\widetilde}

\def\Ric{{\rm Ric\,}}

\begin{document}

\title{Chern's magic form and the Gauss-Bonnet-Chern mass}

\author{Guofang Wang}
\address{ Albert-Ludwigs-Universit\"at Freiburg,
Mathematisches Institut,
Eckerstr. 1
D-79104 Freiburg, Germany}
\email{guofang.wang@math.uni-freiburg.de}

\author{Jie Wu}
\address{Department of Mathematics, Zhejiang University, Hangzhou 310027, P. R. China}
\email{wujiewj@zju.edu.cn}

\begin{abstract} In this note, we use Chern's magic  form $\Phi_k$ in his famous proof of the Gauss-Bonnet theorem
to define a  mass for asymptotically flat manifolds. It turns out that the new defined mass is  equivalent to  the one that we introduced recently by using
the Gauss-Bonnet-Chern curvature $L_k$.
Moreover, this equivalence implies a simple proof of the equivalence between the ADM mass and the intrinsically defined mass via the Ricci tensor,
which was reconsidered by Miao-Tam \cite{MT} and Herzlich \cite{H} very recently.
\end{abstract}

\thanks{J. Wu is supported by NSF of China under Grant No. 11401553.}



\maketitle

%\tableofcontents

\
%\chapter{chapter}

\section{Introduction}
The ADM mass plays an important role in the Einstein gravity. It is a conserved quantity for asymptotically flat manifolds.  A complete  manifold $(M^n,g)$ is said to be
asymptotically flat (AF) of decay order $\tau$ (with one end) if there exists a compact set $K$ such that
$M\setminus K$ is diffeomorphic to  $\mathbb{R}^n\setminus B_R(0)$ for some $R>0$ and in the standard coordinates of $\mathbb{R}^n$, the metric $g$ has the following expansion
 $$g_{ij}=\delta_{ij}+\sigma_{ij},$$
 with $$|\sigma_{ij}|+ r|\partial\sigma_{ij}|+ r^2|\partial^2\sigma_{ij}|=O(r^{-\tau}),$$
 where $r$ and $\partial$ denote the Euclidean distance and the standard derivative operator on $\mathbb{R}^n$ with the standard metric $\delta$, respectively. If the scalar
 curvature $R$ is integrable on $(M^n,g)$ and $\tau>\frac{n-2}{2}$, the ADM mass \cite{ADM} is defined by
\begin{equation}\label{m1}
m_{ADM}:=\frac 1{2(n-1)\omega_{n-1}}\lim_{r\to\infty}\int_{S_r}(g_{ij,i}-g_{ii,j})\nu_j^{\delta} d\sigma^{\delta},
\end{equation}
where $\omega_{n-1}$ is the volume of $(n-1)$-dimensional standard unit sphere and  $S_r$ is the Euclidean coordinate sphere, $d\sigma^{\delta}$ is the volume element on $S_r$ induced by
the Euclidean metric, $\nu^{\delta}$ is the outward unit normal vector to $S_r$ in $\mathbb{R}^n$ and $g_{ij,k}= \partial_kg_{ij}$
are  the ordinary partial derivatives. The well-definedness and invariance of $m_{ADM}$ was proved by Bartnik \cite{Bar} (see also the work by Chru\'sciel \cite{Ch1}).
It is known that there is alternative formulation of the ADM mass via the Einstein tensor and a radial direction Euclidean conformal Killing field  (see  Ashtekar-Hasen \cite{AH} and
Chru\'sciel \cite{Ch2}).  Precisely,
\begin{equation}\label{mI}
m_I=-\frac{1}{(n-1)(n-2)\omega_{n-1}}\lim_{r\to\infty}\int_{S_r}(Ric-\frac 12 Rg)(X,\nu)d\sigma^g,
\end{equation}
where $X$ denotes the Euclidean conformal Killing vector field $r\partial r$, $d\sigma^g$ is the volume element on $S_r$ induced by the metric
$g$ and $\nu$ is the outward unit normal vector to $S_r$ in $(M^n, g)$. The equivalence of the two masses was proved by applying a density theorem to reduce the general case to the
harmonic asymptotics case (see the work of Huang \cite{Huang} for instance). Very recently, there appeared new proofs of the equivalence by Miao-Tam \cite{MT} by calculating directly
in coordinates and by Herzlich \cite{H} in a coordinate-free way. One of the aims of this note is to give another, simpler proof of this equivalence.

As a generalization of the ADM mass, recently we introduced a higher order mass by using
the Gauss-Bonnet-Chern curvature $L_k$, which is a natural generalization of the scalar
curvature, as follows:
\begin{defi}\label{GBC defi}
$~$\cite{GWW}
 Let $n>2k$. Suppose that $(M^n,g)$ is an asymptotically flat of decay order
 $\tau>\frac{n-2k}{k+1}$ and the Gauss-Bonnet curvature  $L_k$ is integrable on $(M^n,g)$. The
 Gauss-Bonnet-Chern mass (or GBC mass) given by
 \begin{equation}\label{GBC0}
m_k=m_{GBC}:=\frac{(n-2k)!} {2^{k-1}(n-1)!\,\omega_{n-1}}
\lim_{r\to\infty}\int_{S_r}P_{(k)}^{ijls}\partial_s g_{jl} \nu_{i}^{\delta}d\sigma^{\delta},
\end{equation}
exists and does not depend on the choice of coordinates.
\end{defi}
For the definition of the Gauss-Bonnet-Chern curvature $L_k$ and
the 4-tensor $P_{(k)}$, see Section 2 below.

Analog to the work of Ashtekar-Hasen and Chru\'sciel, we can also introduce a higher order intrinsically defined mass,
by using a generalized Einstein tensor ${\mathcal E}^{(k)}$, which we will call it the Lovelock curvature, since Lovelock gave a characterization of such tensors in \cite{Lo}.  For
its definition see \eqref{2} in Section 2 below.

\begin{defi}Let $X$ be the radial direction Euclidean conformal Killing vector field $r\frac{\p}{\p r}$, we define
\begin{equation}\label{Ric_m_k}m^k_{I}:=-\frac{(n-2k-1)!}{2^{k-1}(n-1)!\,\om_{n-1}}\lim_{r\to \infty} \int_{S_r}  {\mathcal E}^{(k)}(X , \nu)d\sigma^g,
\end{equation}
whenever this limit is convergent.
\end{defi}
This is the same as $m_I$ when $k=1$, since ${\mathcal E}^{(1)}$ is exactly the ordinary Einstein tensor $Ric-\frac 12 R g$.
In this note, we would like to see if
it is equivalent to the GBC mass $m_{GBC}$ for general $1\leq k<\frac n2$.

On the other hand,  the Gauss-Bonnet-Chern curvature $L_k$ looks like the magic form $\Phi_k$, which was introduced by Chern \cite{Chern1} in his famous intrinsic proof of the
Gauss-Bonnet theorem.
Its definition is recalled in \eqref{eq3.1} below.  Therefore it is natural to ask if one can use $\Phi_k$ alternatively to define a mass.

By a slight modification, we show that one can really do it.
\begin{defi}
Under the conditions in Definition \ref{GBC defi}, we define a new mass
\begin{equation}\label{eq_chern_mass}
m^C_k=\frac{1}{2^k(n-1)!\,\om_{n-1}}\lim_{r\to \infty} \int_{S_r} r^{n-2k}\nu^*(\Phi_k),
\end{equation}
where the $(n-1)$ form $\Phi_k$ is defined in \eqref{eq3.1} below and the outer unit  normal $\nu$ is viewed as a map from $S_r$ to the sphere bundle over $M$.
\end{defi}
The well-definedness and the geometric invariance of this quantity can be checked directly and one can refer to Section 4 for more details. We call it  {\it Chern mass}, or the $k^{th}$
Chern mass.
A more interesting point is  that this mass has a simple relation to the GBC mass $m_{GBC }$ as well as to the intrinstrically  defined mass $m_I^k$. To be more precisely, we have
\begin{enumerate}
\item [(i)]$d(r^{n-2k} \Phi_k)
=(n-2k)!\,L_k *1+O(r^{-(k+1)\tau-2k}).$  (Lemma \ref{lem1})
\item  [(ii)]$r^{n-2k}\nu^*(\Phi_k)=-2\cdot(n-2k-1)!\, {\mathcal E}^{(k)}(r\frac \p {\p r} ,\frac \p {\p r})d\sigma^g+ O(r^{-(k+1)\tau-2k+1})$. (Lemma \ref{lem2})
\end{enumerate}

\noindent From (ii) it is easy to see that $m_k^C=m_I^k$. From (i) we can prove  $m_k^C=m_{GBC},$ by applying a trick used recently by Herzlich in \cite{H}. This trick was also used by Schoen
in \cite{Schoen}. Therefore, we have

\begin{thm}
If $n>2k$ and $(M^n,g)$ is an asymptotically flat of decay order $\tau>\frac{n-2k}{k+1}$ with integrable Gauss-Bonnet curvature  $L_k$  on $(M^n,g)$. Then we have the equivalence of the three masses
$$m_{GBC}=m_{I}^k=m_k^C.$$
\end{thm}
When $k=1$ this result also provides
a simple proof of the equivalence between the ADM mass  $m_{ADM}$ and  $m_I$, mentioned above.

We remark that from our work, one can see that the mass can be seen as a generalization of Gauss-Bonnet-Chern theorem for asymptotically flat manifolds, or more precisely, as a
renormalized Gauss-Bonnet-Chern theorem. See Remark \ref{rem_add} below.

In addition to the mass,  it is also interesting to define other invariants, such as the center of mass for these higher order masses. Related to the ADM mass $m_{ADM}$, under some additional parity condition at infinity,  the Hamiltonian formulation of center of mass was proposed by Regge-Teitelboim \cite{RT} (See also the work of Beig-\'O Murchadha \cite{BO}). Similar to (\ref{mI}), it is also known that there is an alternative formulation of a center of mass suggested by R. Schoen \cite{Huang} using the Einstein tensor and Euclidean conformal Killing vector fields.  Their equivalence has been first proved by Huang \cite{Huang} by applying a density theorem and very recently by Miao-Tam \cite{MT} and Herzlich \cite{H} via different methods.
From our work, one can also define a corresponding center for  our masses and show their equivalence. We leave this to the interested reader.

The rest of the paper is organized as follows: In Section 2, we recall some definitions as well as the properties, including the Gauss-Bonnet-Chern curvature, the GBC mass and the
generalized Einstein tensor. Chern's magic forms $\Phi_k$ are reviewed in Section 3 and we apply it to define the Chern mass in Section 4. Section 5 is devoted to the equivalence of
$m_{GBC}$ and $m_k^C$. In Section 6, we show that $m_{I}^k=m_k^C.$

\section{The Gauss-Bonnet-Chern curvature and mass}
Recall that the Gauss-Bonnet-Chern curvature is given by
 \begin{equation}\label{Lk2}
L_k:=\frac{1}{2^k}\d^{i_1i_2\cdots i_{2k-1}i_{2k}}
_{j_1j_2\cdots j_{2k-1}j_{2k}}{R_{i_1i_2}}^{j_1j_2}\cdots
{R_{i_{2k-1}i_{2k}}}^{j_{2k-1}j_{2k}},
\end{equation}
where $\d^{i_1i_2\cdots i_{2k-1}i_{2k}}
_{j_1j_2\cdots j_{2k-1}j_{2k}}$ is the generalized Kronecker delta defined by
\begin{equation}\label{generaldelta}
 \d^{j_1j_2 \cdots j_r}_{i_1i_2 \cdots i_r}=\det\left(
\begin{array}{cccc}
\d^{j_1}_{i_1} & \d^{j_2}_{i_1} &\cdots &  \d^{j_r}_{i_1}\\
\d^{j_1}_{i_2} & \d^{j_2}_{i_2} &\cdots &  \d^{j_r}_{i_2}\\
\vdots & \vdots & \vdots & \vdots \\
\d^{j_1}_{i_r} & \d^{j_2}_{i_r} &\cdots &  \d^{j_r}_{i_r}
\end{array}
\right),
\end{equation}
and ${R_{ij}}^{sl}$ is the Riemannian curvature tensor in the local coordinates. One can easily check that $L_1$ is just the
scalar curvature $R$. When $k=2$, it is the (second) Gauss-Bonnet curvature
\[L_2 = \|Rm\|^2-4\|Ric\|^2+R^2,\]
which appeared at the first time in the paper of Lanczos \cite{Lan} in 1938.

In the definition of the Gauss-Bonnet-Chern mass \cite{GWW}, the key observation is that the Gauss-Bonnet curvature has the
following decomposition
$$L_k=P^{ijls}_{(k)} R_{ijls},$$
with the crucial property that $P_{(k)}$ is divergence-free, i.e.,
\begin{equation}\label{eq_a1}
 \nabla_iP^{ijls}_{(k)}=0,
\end{equation}
and has the same symmetry as the curvature tensor.
Here the 4-tensor $P_{(k)}$ is defined by
\begin{equation}
P^{ijls}_{(k)}:=\frac{1}{2^k}\d^{i_1i_2\cdots i_{2k-3}i_{2k-2}ij}
_{j_1j_2\cdots  j_{2k-3}j_{2k-2} j_{2k-1}j_{2k}}{ R_{i_1i_2}}^{ j_1 j_2}\cdots
{R_{i_{2k-3}i_{2k-2}}}^{ j_{2k-3}j_{2k-2}}g^{j_{2k-1}l}g^{j_{2k}s}.
\end{equation}
For asymptotically flat manifolds, one has the following asymptotic expansion for $L_k$ \cite{GWW},
\begin{equation}\label{eq1.3}
L_k =2\partial_i(P_{(k)}^{ijls} \partial_s g_{jl}) +O(r^{-(k+1)\tau-2k}),
\end{equation}
which suggests the definition of the GBC mass in \cite{GWW} by
\begin{equation}\label{GBC}
m_k=m_{GBC}:=c(n,k)
\lim_{r\to\infty}\int_{S_r}P_{(k)}^{ijls}\partial_s g_{jl} \nu_{i}^{\delta}d\sigma^{\delta},
\end{equation}
with
\[c(n,k)=\frac{(n-2k)!} {2^{k-1}(n-1)!\,\omega_{n-1}},\]
provided $L_k$ is integrable on $(M^n,g)$ and $\tau>\frac{n-2k}{k+1}$. Here the constant $c(n,k)$ is determined by calculating the mass of the
Schwarzschild-type solution in the Gauss-Bonnet gravity to obtain the expected answer. One can easily see that $m_1$ is exactly the ADM mass $m_{ADM}$.

The  Einstein-like tensor associated  to $L_k$ is introduced and characterized  by Lovelock \cite{Lo} by
\begin{equation}\label{2}
\mathcal{E}^{(k)i}{}_j = -{1\over 2^{k+1}}\d^{ii_1i_2\cdots i_{2k-1}i_{2k}}
_{jj_1j_2\cdots j_{2k-1}j_{2k}}{R_{i_1i_2}}^{j_1j_2}\cdots
{R_{i_{2k-1}i_{2k}}}^{j_{2k-1}j_{2k}}.
\end{equation}
As a convention, we set $\mathcal{E}^{(0)}=g$. When $k=1$, then $\mathcal{E}^{(1)}$ is just the usual Einstein tensor $E:=\Ric-\frac 12 R_g g$. Like the usual Einstein tensor $E$,
the
generalized Einstein-like tensor $\mathcal{E}^{(k)}$ satisfies a conversation law, namely, $\div\mathcal{E}^{(k)}=0$, i.e., $\nabla_j\mathcal{E}^{(k)}_{ij}=0$. $\mathcal{E}^{(k)}$ is
called the $k$-th Lovelock curvature and one can see
easily that ${\rm Trace}\,\; \mathcal{E}^{(k)}=-\frac{n-2k}{2}L_k$.

As indicated in the Introduction, with this generalized Einstein tensor ${\mathcal E}^{(k)}$,
we can introduce a higher order intrinsically defined mass $m_{I}^k$ by
\begin{eqnarray*}
m^k_{I}&:=&-\frac{(n-2k-1)!}{2^{k-1}(n-1)!\,\om_{n-1}}\lim_{r\to \infty} \int_{S_r}  {\mathcal E}^{(k)}(X , \nu)d\sigma^g\\
&=&-\frac{(n-2k-1)!}{2^{k-1}(n-1)!\,\om_{n-1}}\lim_{r\to \infty} \int_{S_r}  {\mathcal E}^{(k)}(r\frac\p {\p r} , \frac \p {\p r})d\sigma^g,
\end{eqnarray*}
since $\nu=\frac{\p}{\p r}+O(r^{-\tau})$ for AF manifolds of decay order $\tau>\frac {n-2k}{k+1}.$

One can also write $L_k$  in terms of differential forms as follows.
\begin{eqnarray*}
L_k*1&=&\frac{1}{2^k}\delta^{\a_1\a_2\cdots\a_{2k}}_{\b_1\b_2\cdots\b_{2k}}{R_{\a_1\a_2}}^{\b_1\b_2} \cdots {R_{\a_{2k-1}\a_{2k}}}^{\b_{2k-1}\b_{2k}}\om_{1}\wedge \cdots \wedge
\om_{n}\\
&=&\frac{1}{(n-2k)!}\e^{\a_1\a_2 \cdots \a_n}  \Omega_{\a_1\a_2}\wedge  \cdots
 \wedge
 \Omega_{\a_{2k-1}\a_{2k}} \wedge \om_{\a_{2k+1}}\wedge \cdots \wedge \om_{\a_n},
 \end{eqnarray*}
 where $\e^{\a_1\a_2 \cdots \a_n}:=\delta^{\a_1\a_2 \cdots \a_n}_{1~2~\cdots~ n}$, $\{\om_\alpha\}_{\alpha=1}^n$ is the dual coframe of a local frame $\{e_\alpha\}_{\alpha=1}^n$, $*1$
 is the volume form and $\Omega_{\a_1\a_2}$ is the curvature two-form given by
 $$\Omega_{\a_1\a_2}=  \frac 12{R_{\a_1\a_2}}^{\b_1\b_2} \om_{\b_1}\wedge \om_{\b_2}.$$
 Define a two form $Q$ from the $P_{(k)}$ curvature tensor by
 $$Q^{\a\b}=P_{(k)}^{\a\b\gamma\delta} \om_\gamma\wedge \om_\delta.$$
 It is easy to check that \eqref{eq_a1} is equivalent to
 \begin{equation}\label{eq_a2}
d^*Q^{\a\b}=0,
\end{equation}
where $d^*$ is the dual operator of the differential operator $d$. The Hodge dual operator of $Q^{\a\b}$ is
\begin{eqnarray}\label{add_c}
  (*Q^{\a_1\a_2})&=&\frac 1{(n-2k)!} \e^{\a_1\a_2 \cdots \a_n} \Omega_{\a_{3}\a_{4}} \wedge  \cdots  \wedge
 \Omega_{\a_{2k-1}\a_{2k}} \wedge \om_{\a_{2k+1}}\wedge \cdots \wedge \om_{\a_n}\nonumber\\
 &=& \frac1{2^{k-1}\cdot(n-2k)!}\e^{\a_1\a_2 \cdots \a_n}
 {R_{\a_3\a_4}}^{\b_3\b_4} \cdots {R_{\a_{2k-1}\a_{2k}}}^{\b_{2k-1}\b_{2k}} \om_{\b_3}\cdots \wedge  \om_{\b_{2k}}\nonumber\\
 &&\wedge\om_{\a_{2k+1}}\wedge \cdots \wedge \om_{\a_n}.
 \end{eqnarray}
It is clear from \eqref{add_c} that
$$L_k\ast 1=\Omega_{\a_1\a_2}\wedge (*Q^{\a_1\a_2}).$$
 Let $\om_{\a\b}$ be the Levi-Civita connection one-form with respect to the local frame $\{e_\alpha\}_{\alpha=1}^n$. Now we can rewrite \eqref{eq1.3} in terms of differential forms
\begin{equation}\label{eq1.4}
L_k*1=d(\om_{\a_1\a_2}\wedge(*Q^{\a_1\a_2} ))+O(r^{-(k+1)\tau-2k}),
\end{equation}
which is the formula used in \cite{LN}. In fact it follows readily from \eqref{eq_a2},
if one computes $ d(\om_{\a_1\a_2}\wedge(*Q^{\a_1\a_2}))$ directly:
\begin{equation}\label{eq2.7} \begin{array}{rcl}
    d(\om_{\a_1\a_2}\wedge(*Q^{\a_1\a_2})) &=& d\om_{\a_1\a_2}\wedge(*Q^{\a_1\a_2} ) \\
   &=& (\Omega_{\a_1\a_2}+ \om_{\a_1\b}\wedge \om_{\b \a_2})\wedge(*Q^{\a_1\a_2}) \\
   &=& L_k* 1+  \om_{\a_1\b}\wedge \om_{\b \a_2}\wedge (*Q^{\a_1\a_2})\\
   &=& L_k*1+O(r^{-(k+1)\tau-2k}).
\end{array}
\end{equation}
In the first equality, we have used \eqref{eq_a2}.
We remark that unlike $L_k*1$, the $n-1$ form $\om_{\a_1\a_2}\wedge(*Q^{\a_1\a_2}) $  does depend on the frame. Or  in other words,
$\om_{\a_1\a_2}\wedge(*Q^{\a_1\a_2}) $ is an $n-1$ form on the frame bundle over $M$.
Therefore, when we use the expression $P_{(k)}^{ijls} \partial_s g_{jl}$, or
equivalently using $\om_{\a_1\a_2}\wedge(*Q^{\a_1\a_2} )$ to define a mass, we need to check
that the defined mass does not depend on the choice of the frames.

The mass $m_{GBC}$ defined above
 is trivially the same if one uses $ \om_{\a_1\a_2}\wedge(*Q^{\a_1\a_2})$ to define a mass, as in \cite{LN}.

The GBC mass was first studied in \cite{GWW} and its positivity was proved in \cite{GWW} for graphic manifolds and in \cite{GWW2} for conformally flat manifolds.
The GBC mass for higher codimensional graphs was studied in \cite{LWX} and \cite{GM}.

\section{Chern's magic forms}

To prove the Gauss-Bonnet formula for a general closed Riemannian manifold ${M}^n$,
Chern \cite{Chern1, Chern2} turn to consider the sphere bundle $\mathcal{S}({M}):=
\{(p,v):p\in{M}, v\in T_p(M) \text{ and }|v|=1 \}$ of dimension $2n-1$. More precisely, he introduced the following
important forms:
\begin{align}\label {eq3.1}
\Phi_k &=
\e^{\a_1\a_2 \cdots \a_{n-1}}  \Omega_{\a_1\a_2} \wedge  \cdots  \wedge
 \Omega_{\a_{2k-1}\a_{2k}} \wedge \om_{\a_{2k+1}n}\wedge \cdots \wedge \om_{\a_{n-1}n},\\\label{eq3.2}
 \Psi_k &=
 2(k+1) \e^{\a_1\a_2 \cdots \a_{n-1}}  \Omega_{\a_1\a_2}  \wedge  \cdots  \wedge
 \Omega_{\a_{2k-1}\a_{2k}} \wedge  \Omega_{\a_{2k+1}n} \wedge \om_{\a_{2k+2}n}\wedge \cdots \wedge \om_{\a_{n-1}n},
\end{align}
where $e_n=\nu$ denotes the unit outer normal vector field of $S_r$ and $\e^{\a_1\a_2 \cdots \a_{n-1}}
=\delta^{\a_1\a_2 \cdots \a_{n-1}}_{12\cdots n-1}$ as before. Note that $\Phi_k$ are $n-1$ forms, $\Psi_k$
are $n$ forms and if $n$ is even, $\Psi_{\frac{n}{2}-1}$ equals the Paffian $\Omega=L_{[\frac{n}{2}]}*1$ . More importantly, it was proved by Chern \cite{Chern1, Chern2} that
$\Omega$ is an exact form (i.e. it is the exterior derivative of a $n-1$ form) by observing the following
iteration relation between these forms:
\begin{equation}\label{eq3.3}
d \Phi_k=\Psi_{k-1}+\frac {n-2k-1}{2(k+1)} \Psi_k,\quad k=0,1,\cdots,\left[\frac n2\right]-1,
\end{equation}
where $[\frac n2]$ denotes the largest integer $\le n/2$.  From \eqref{eq3.3} Chern obtained
\begin{equation}\label{eq3.4}
L_{\frac n2} *1=d\Pi,
\end{equation}
{ where}
\begin{equation*}\Pi:=(-1)^{\frac n2-1}\sum_{k=0}^{\frac n2 -1}(-1)^k \frac{(\frac n2-k-1)!\,(\frac n2)!}{(n-2k-1)!\,k!\,2^{2k-n+1}}\Phi_k,
\end{equation*}
when $n$ is even. Therefore, the Euler density $L_{\frac n2} *1$ is an ``exact form" with $\Pi$, a form on the sphere bundle. The intrinsic proof of the Gauss-Bonnet Theorem
follows from \eqref{eq3.4} and the Poincar\'e-Hopf index theorem.
See also a nice survey \cite{Weiping}.

If one compares \eqref{eq2.7} with \eqref{eq3.3}, one can see that they look quite similar.
It should be noticed that $\Phi_k$  ($\Psi_k$ resp.)   are $n-1$ forms  ($n$ forms resp.) defined on the sphere bundle, while
$L_k*1$ is an $n$ form defined on $M$ and $\om_{\a_1\a_2}\wedge(*Q)^{\a_1\a_2} $($\om_{\a_1\b}\wedge \om_{\b \a_2}\wedge(*Q)^{\a_1\a_2} $ resp.)
are $n-1$ forms($n$ forms resp.) on the frame bundle.

Due to this similarity, a natural question arises: can one define an invariant by using $\Phi_k$ for
asymptotically flat manifolds?

\section{A mass defined by using $\Phi_k$}
In this section we modify  the Chern form $\Phi_k$ slightly  to define a mass for asymptotically flat manifolds.

Let $S_r$ be a coordinate sphere of radius $r$ in the asymptotically flat manifolds $(M^n, g)$, for large $r$.
Let $\nu$ be the outerward unit normal vector field along $S_r$.  Viewing $\nu$ as a map from $S_r$ to the sphere bundle over $M$,
we consider the pull-back form $\nu^*(\Phi_k)$ and  define a quantity by
$$\tilde m_k=\lim_{r\to \infty} \int_{S_r} r^{n-2k}\nu^*(\Phi_k),$$
whenever this limit is convergent.

First, we need the following decay estimates:
\begin{lem} \label{lem0}
On an asymptotically flat manifold $(M^n, g)$ with decay order $\tau$, we have
 $\om_{jn}=-\frac 1 r \om_{j}+ O( r^{-1-\tau})$ for all $j=1,2,\cdots, n-1$.
\end{lem}
\begin{proof}
  First we choose  $e_n=\nu=\frac{\nabla_g r}{|\nabla_g r|}=\frac{\partial}{\partial r}+O(r^{-\tau})$, which implies
  $$\om_n=dr+O(r^{-\tau}).$$
In view of the first structure equation $d\om_n=\om_j\wedge \om_{jn}$, one can derive
  $$\om_{jn}=h_{ij}\om_i+O(r^{-1-\tau}),\quad i=1,\cdots, n-1.$$
The conclusion follows by noting that $h_{ij}=-\frac{\delta_{ij}}{r}+O(r^{-1-\tau}).$
\end{proof}

With the above lemma, it is crucial to observe the following
\begin{lem} \label{lem1}
\begin{equation}\label{key}
d(r^{n-2k} \Phi_k)
=(n-2k)!\,L_k *1+O(r^{-(k+1)\tau-2k}).
\end{equation}
\end{lem}

\begin{proof} By using the Chern formula \eqref{eq3.3} together with \eqref{eq3.1} and \eqref{eq3.2},  we have
\begin{eqnarray*}
d(r^{n-2k} \Phi_k) &=& r^{n-2k} d\Phi_k+(n-2k)r^{n-2k-1} dr \wedge \Phi_k\\
&=&r^{n-2k}\left(\Psi_{k-1}+\frac {n-2k-1}{2(k+1)}\Psi_k\right)+(n-2k)r^{n-2k-1}dr\wedge \Phi_k\\
&=&r^{n-2k}\left(\Psi_{k-1}+(n-2k)\frac 1 r dr\wedge \Phi_k\right)+O(r^{-(k+1)\tau-2k})\\
&=& r^{n-2k}
\e^{\a_1\cdots\a_{n-1}}\Big(2k \Omega_{\a_1\a_2}\wedge \cdots \wedge \Omega_{\a_{2k-1}n}\wedge \om_{\a_{2k}n}\wedge\cdots \wedge  \om_{\a_{n-1}n} \\
&&+(n-2k) \frac 1 r  dr \wedge \Omega_{\a_1\a_2} \wedge  \cdots  \wedge
 \Omega_{\a_{2k-1}\a_{2k}} \wedge \om_{\a_{2k+1}n}\wedge \cdots \wedge \om_{\a_{n-1}n}\Big)\\
 &&+O(r^{-(k+1)\tau-2k}).
\end{eqnarray*}
On the other hand, using $\om_{jn}=-\frac{\om_j}{r}+O(r^{-1-\tau}) $ ($j=1,\cdots, n-1$) and $\om_n=dr+O(r^{-\tau})$  (Lemma \ref{lem0}), we have

 \begin{eqnarray*}
(n-2k)!\,L_k*1 &=& \e^{\a_1\a_2 \cdots \a_n}
 \Omega_{\a_1\a_2}\wedge \cdots \wedge \Omega_{\a_{2k-1}\a_{2k}}\wedge \om_{\a_{2k+1}}\wedge\cdots \wedge  \om_{\a_{n}}\\
  &=& \e^{\a_1\a_2 \cdots \a_{n-1}}
  \Big(2k\cdot(-1)^n\Omega_{\a_1\a_2}\wedge \cdots \wedge \Omega_{\a_{2k-1}n}\wedge \om_{\a_{2k}}\wedge\cdots \wedge  \om_{\a_{n-1}} \\
  &&+(n-2k)\cdot(-1)^{n-1} dr \wedge \Omega_{\a_1\a_2} \wedge  \cdots  \wedge
 \Omega_{\a_{2k-1}\a_{2k}} \wedge \om_{\a_{2k+1}}\wedge \cdots \wedge \om_{\a_{n-1}}\Big)\\
 &&+O(r^{-(k+1)\tau-2k})\\
&=& r^{n-2k}
\e^{\a_1\cdots\a_{n-1}}\Big(2k \Omega_{\a_1\a_2}\wedge \cdots \wedge \Omega_{\a_{2k-1}n}\wedge \om_{\a_{2k}n}\wedge\cdots \wedge  \om_{\a_{n-1}n} \\
&&+(n-2k) \frac 1 r  dr \wedge \Omega_{\a_1\a_2} \wedge  \cdots  \wedge
 \Omega_{\a_{2k-1}\a_{2k}} \wedge \om_{\a_{2k+1}n}\wedge \cdots \wedge \om_{\a_{n-1}n}\Big)\\
 &&+O(r^{-(k+1)\tau-2k}).
 \end{eqnarray*}
Thus we complete the proof.
\end{proof}
\begin{rema}\label{rem_add} Formula \eqref{key} means that on an asymptotically flat manifold, up to a higher order term,
 $L_k*1$ is an ``exact form" with $r^{n-2k}\Phi_k$, a form on the sphere bundle. In spirit, it is very similar to \eqref{eq3.4}.

\end{rema}

From Lemma \ref{lem1}, we can define a mass by using the form $\Phi_k$.
\begin{defi} Let $(M,g)$ be an asymptotically flat  manifold of decay order $\tau>\frac{n-2k}{k+1}$ and with $L_k\in L^1(M,g)$.
For each given integer $1\le k< \frac n2$ we define $k^{th}$-Chern mass by
\begin{equation}
m^C_k=\frac{1}{2^k(n-1)!\,\om_{n-1}}\lim_{r\to \infty} \int_{S_r} r^{n-2k}\nu^*(\Phi_k).
\end{equation}
\end{defi}

Lemma \ref{lem1} implies that the limit in $m_k^C$ is convergent. Moveover, one can prove that

\begin{lem}$m_k^C$ is well-defined and a geometric invariant, provided that $L_k$ is integrable and $\tau>\frac{n-2k}{2k+1}$.

\end{lem}
\begin{proof}
 In view of Lemma \ref{lem1}, one can show this lemma by using methods of Bartnik \cite{Bar} as in \cite{GWW,LN}. See also the approach due to Michel \cite{M}. We skip the proof
 here, since
 the result also follows from the equivalence proved below.
\end{proof}

\section{Equivalence of the GBC mass and the  Chern mass}

In this section, we show that the GBC mass introduced in \cite{GWW} is the same as the Chern mass introduced in the previous section.

\begin{thm}
 \label{thm1}
 Let $(M,g)$ be an asymptotically flat  manifold of decay order $\tau>\frac{n-2k}{k+1}$ and with $L_k \in L^1(M,g)$.
 Then we have
 $$m_k^C=m_{GBC}.$$
\end{thm}
\begin{proof}
 From the previous section we know that $L_k$ has two different expansions:
 \begin{eqnarray}\label{eq_a1'}
L_k *1&=&  2 \partial_i(P_{(k)}^{ijls}\partial_s g_{jl}) *1+O(r^{-(k+1)\tau-2k})\\
&=&\frac1{(n-2k)!} d(r^{n-2k}\Phi_k)+O(r^{-(k+1)\tau-2k}).\label{eq_a2'}
 \end{eqnarray}
If we define a one-form $a$ by $a={P_{(k)i}}^{jsl}\partial_l g_{js} dx^i$, then \eqref{eq_a1'}  is rewritten as
\begin{align}\label{eq_a3'}
  L_k *1=  2 d(*a)+O(r^{-(k+1)\tau-2k}).
  \end{align}
  It is clear that $2*a= (\om_{\a_1\a_2}\wedge(*Q^{\a_1\a_2})) $. See \eqref{eq2.7}.
The GBC mass $m_{GBC}$ can certainly be defined by using  $\lim_{r\to \infty}\int_{S_r} 2*a$ as well.

Now we apply a trick used
by Herzlich recently in \cite{H}. For any large $r>4$, we consider a modified metric $h$ by
gluing the Euclidean metric $\d$
inside the ball $B_{\frac r4}$  and the original metric $g$ outside the ball
$B_{r}$ as follows: Let $\eta:\R^n\to [0,1]$ be a cut-off function satisfying $\eta=0$ in $B_{\frac 12}$ and $\eta=1$ outside $B_{\frac 34}$. Set
$\eta_r(x)=\eta(rx)$. It is clear that
$\eta_r$ satisfies
$$ r|\n \eta_r|+r^2|\n^2 \eta_r|+ r^3 |\n^3 \eta_r|\le c_0,$$
for some universal constant $c_0>0$, which is independent of $r>>1$. We then define for each
$r>>1$ a metric on the annulus $A=A(\frac 14 r, r)$:
$$
h:=h_r
:=
\eta_r \cdot g
+
(1-\eta_r)\cdot\d.$$
It is clear that $h$ is also an asymptotically flat metric of decay order $\tau$ with uniform estimates independent of $r>>1$.
Hence, for metric $h$, \eqref{eq_a1'}--\eqref{eq_a3'} hold.
Note that $*a(h)=\Phi_k(h)=0$ in $B_{\frac 14 r}$ and $*a(h)=*a(g)$, $\Phi_k(h)=\Phi_k(g)$ outside $B_{\frac 34 r}$.
Hence we infer from (\ref{eq_a2'}) and (\ref{eq_a3'}) that
\begin{eqnarray*}
 2\int_{S_r} *a(g)=2 \int_{S_r} *a(h)&=&2 \int_{A(\frac 14 r,r)} d(*a) (h)
 \\
 &=&   \int_{A(\frac 14 r,r)} L_k(h)*1 +\int_{A(\frac 14 r,r)}O(r^{-(k+1)\tau-2k})\\
 &=&\int_{A(\frac 14 r,r)}L_k(h)*1 +o(1),
 \end{eqnarray*}
 and
 \begin{eqnarray*}
\frac1{(n-2k)!} \int_{S_r} r^{n-2k} \Phi_k(g)&=&\frac1{(n-2k)!}  \int_{S_r} r^{n-2k} \Phi_k(h)
  =\frac1{(n-2k)!}  \int_{A(\frac 14 r,r)} d (r^{n-2k}\Phi_k(h))
 \\
 &=&   \int_{A(\frac 14 r,r)} L_k(h)*1 +o(1).
 \end{eqnarray*}
 It follows that
 $$  2\int_{S_r} *a(g)=  \frac1{(n-2k)!} \int_{S_r} r^{n-2k} \Phi_k(g)+ o(1),$$
 and hence the conclusion as $r\rightarrow\infty$.
\end{proof}

\begin{rema}  To prove that $m_{GBC}=m_k^C$, one may would like to show that $*a=\frac1{2\cdot(n-2k)!}\cdot r^{n-2k}\Phi_k+o(r^{-(n-1)})$ pointwisely.
However, this  should not be true, since
$r^{n-2k}\Phi_k$ consists of $k$ times curvature tensors, while $*a$ has only $k-1$ times curvature tensors together with a connection form.
\end{rema}

\section{Equivalence between $m_I^k$ and $m^C_k$}

Now we want to see what exactly $r^{n-2k}\nu^*(\Phi_k)$ is. Notice that  the induced area element of $S_r$ from the volume form of $g$ is
\begin{equation}\label{eq6.0} d\sigma^g=(-1)^{n-1}\om_1\wedge\cdots \wedge \om_{n-1}.\end{equation}

We first consider the $k=1$ case. By Lemma \ref{lem0},
\begin{eqnarray*}
r^{n-2}\nu^*(\Phi_1) &=&r^{n-2}
\e^{\a_1\a_2 \cdots \a_{n-1}}  \Omega_{\a_1\a_2} \wedge   \om_{\a_{3}n}\wedge \cdots \wedge \om_{\a_{n-1}n}\\
 &=& (-1)^{n-1} r\e^{\a_1\a_2 \cdots \a_{n-1}}  \Omega_{\a_1\a_2} \wedge   \om_{\a_{3}}\wedge \cdots \wedge \om_{\a_{n-1}}+O(r^{-1-2\tau}).
 \end{eqnarray*}

 One can check directly that (or one can refer to the following Lemma 6.2 for a proof of the general $k$ case)
 \begin{equation*}
r^{n-2}\nu^*( \Phi_1)=-2(-1)^{n-1}\cdot(n-3)!\,(\Ric-\frac R 2 g)(r\frac \p {\p r} ,\frac \p {\p r})\om_1\wedge\om_2\wedge\cdots\wedge\om_{n-1}+ O(r^{-1-2\tau}),
 \end{equation*}
 which implies that
  \begin{equation}\label{eq_6.1}
 r^{n-2}\nu^*(\Phi_1)=-2\cdot(n-3)!\,(\Ric-\frac R 2 g)(r\frac \p {\p r} ,\frac \p {\p r}) d\sigma^g+ O(r^{-1-2\tau}).
 \end{equation}
 This means that the integrands in the definition of $m^C_1$ and of  $m_I$ are in fact the same, up to a term of order $O(r^{-1-2\tau})$ which vanishes after integration at infinity. Therefore,
 $m^C_1$ and $m_I$ are trivially the same.
Therefore,
Theorem \ref{thm1} implies the equivalence of the ADM mass $m_{ADM}$ and the intrinsically defined mass $m_{I}$.  Indeed this fact holds for the general $k$.

\begin{thm}
Under the conditions in Definition \ref{GBC defi}, we have
$$m_I^k=m_k^C=m_{GBC}.$$
\end{thm}
As the case $k=1$, the Theorem follows immediately from Theorem \ref{thm1} and the following Lemma.
\begin{lem}\label{lem2}
 \begin{equation}\label{eq_6.2}
  r^{n-2k}\nu^*(\Phi_k)=-2\cdot(n-2k-1)!\, {\mathcal E}^{(k)}\left(r\frac \p {\p r} ,\frac \p {\p r}\right)d\sigma^g+ O(r^{-(k+1)\tau-2k+1}).
 \end{equation}
\end{lem}
\begin{proof}
First by the definition of $\Phi_k$, together with Lemma 4.1, we have
\begin{equation}\label{eq6.3}
\begin{array}{rcl}
r^{n-2k}\nu^*(\Phi_k)&=&\ds\vs r^{n-2k}\e^{\a_1\a_2 \cdots \a_{n-1}}  \Omega_{\a_1\a_2} \wedge  \cdots  \wedge
 \Omega_{\a_{2k-1}\a_{2k}} \wedge \om_{\a_{2k+1}n}\wedge \cdots \wedge \om_{\a_{n-1}n}\\
&=& \ds\vs (-1)^{n-1}r\e^{\a_1\a_2 \cdots \a_{n-1}}  \Omega_{\a_1\a_2} \wedge  \cdots  \wedge
 \Omega_{\a_{2k-1}\a_{2k}} \wedge \om_{\a_{2k+1}}\wedge \cdots \wedge \om_{\a_{n-1}}\\
 \vspace{2mm}
 &&+O(r^{-(k+1)\tau-2k+1})\\
 \vspace{2mm}
&=&(-1)^{n-1}\frac{r}{2^k}\e^{\a_1\a_2 \cdots \a_{n-1}}
 {R_{\a_1\a_2}}^{\b_1\b_2} \cdots {R_{\a_{2k-1}\a_{2k}}}^{\b_{2k-1}\b_{2k}} \om_{\b_1}\wedge\cdots  \wedge \om_{\b_{2k}}\\
 \vspace{2mm}
 &&\wedge
  \om_{\a_{2k+1}}\wedge \cdots \wedge \om_{\a_{n-1}}+O(r^{-(k+1)\tau-2k+1})\\
  \vspace{2mm}
&=&\ds\vs (-1)^{n-1} \frac{r}{2^k}\delta^{\a_1\a_2 \cdots \a_{n-1}}_{12\cdots{n-1}}
 {R_{\a_1\a_2}}^{\b_1\b_2} \cdots {R_{\a_{2k-1}\a_{2k}}}^{\b_{2k-1}\b_{2k}}
 \delta^{1\cdots2k~2k+1\cdots n-1}_{\b_1\cdots\b_{2k}\a_{2k+1}\cdots\a_{n-1}} \\
 \vspace{2mm}
 &&\om_{1}\wedge \om_{2}\cdots  \wedge \om_{n-1}\ds\vs+O(r^{-(k+1)\tau-2k+1})\\
  &=&\ds\vs (-1)^{n-1}\frac{(n-2k-1)!}{2^k}\cdot r \delta^{\a_1 \cdots \a_{2k}}_{\b_1\cdots\b_{2k}}{R_{\a_1\a_2}}^{\b_1\b_2} \cdots {R_{\a_{2k-1}\a_{2k}}}^{\b_{2k-1}\b_{2k}}
  \om_{1}\wedge\cdots  \wedge \om_{n-1}\\
  &&+O(r^{-(k+1)\tau-2k+1}),
  \end{array}
  \end{equation}
where in the last equality, we have used the simple relation
$$\delta^{\a_1\a_2 \cdots \a_{n-1}}_{12\cdots{n-1}}
 \delta^{1\cdots2k~2k+1\cdots n-1}_{\b_1\cdots\b_{2k}\a_{2k+1}\cdots\a_{n-1}}
 ={(n-1-2k)!}\, \delta^{\a_1 \cdots \a_{2k}}_{\b_1\cdots\b_{2k}}.$$
On the other hand, noting $e_n=\frac{\partial}{\partial r}+O(r^{-\tau})$, we compute
\begin{equation}\label{eq6.4}- {\mathcal E}^{(k)}\left(r\frac \p {\p r} ,\frac \p {\p r}\right)=\frac{r}{2^{k+1}}\delta^{\a_1 \cdots \a_{2k}}_{\b_1\cdots\b_{2k}}{R_{\a_1\a_2}}^{\b_1\b_2} \cdots
{R_{\a_{2k-1}\a_{2k}}}^{\b_{2k-1}\b_{2k}} +O(r^{-(k+1)\tau-2k+1}).\end{equation}
Finally, equalities \eqref{eq6.0}, \eqref{eq6.3} and \eqref{eq6.4} imply  the conclusion.
\end{proof}

\noindent{\it Proof of Theorem 6.1.}  Note that $(k+1)\tau+2k-1>n-1$ for $\tau>\frac{n-2k}{k+1}$ and $\nu=\frac{\p}{\p r}+O(r^{-\tau})$, the conclusion follows directly from Theorem
5.1 and Lemma 6.2.

\end{document}